\documentclass{amsart}
\usepackage{amsmath,amssymb}
\usepackage{graphicx,subfigure}
\usepackage{color}

\newtheorem{theorem}{Theorem}

\newtheorem{corollary}[theorem]{Corollary}

\newtheorem{lemma}[theorem]{Lemma}

\theoremstyle{definition}

\theoremstyle{remark}

\newcommand{\al}{\alpha}

\newcommand{\de}{\delta}
\newcommand{\ep}{\epsilon}

\newcommand{\ga}{\gamma}

\newcommand{\la}{\lambda}

\newcommand{\De}{\Delta}

\newcommand{\Si}{\Sigma}
\newcommand{\Om}{\Omega}

\newcommand{\tv}{\widetilde{v}}

\def\RR{\mathbb{R}}
\def\BB{\mathbb{B}}

\newcommand{\cH}{{\mathcal H}}

\newcommand{\pa}{\partial}
\newcommand{\pd}{\partial}
\newcommand\minus\backslash
\newcommand{\id}{{\rm id}}

\newcommand\lan\langle
\newcommand\ran\rangle

\DeclareMathOperator\Div{div}

\renewcommand\leq\leqslant
\renewcommand\geq\geqslant
\newlength{\intwidth}

\begin{document}

\title[Minimal graphs with micro-oscillations]{Minimal graphs with
  micro-oscillations}

\author{Alberto Enciso}
\address{Instituto de Ciencias Matem\'aticas, Consejo Superior de
  Investigaciones Cient\'\i ficas, 28049 Madrid, Spain}
\email{aenciso@icmat.es, mag.ferrero@icmat.es, dperalta@icmat.es}

\author{M\textordfeminine \'Angeles Garc\'\i a-Ferrero}

\author{Daniel Peralta-Salas}

%
%
\begin{abstract}
  We show that there are minimal graphs in $\RR^{n+1}$ whose
  intersection with the portion of the horizontal hyperplane contained
  in the unit ball has any prescribed geometry, up to a small deformation. The proof
  hinges on the construction of minimal graphs that are almost flat
  but have small oscillations whose geometry we can control.
\end{abstract}
\maketitle

\section{Introduction}

Let us consider minimal graphs on the unit ball $\BB^n$ of $\RR^n$.
More precisely, let
$u$ be a function satisfying the equation
\begin{equation}\label{eq:MS}
\Div\bigg(\frac{\nabla u}{\sqrt{1+|\nabla u|^2}}\bigg)=0
\end{equation}
in $\BB^n$. This is equivalent to saying that the graph of~$u$,
\[
\Si_u:=\big\{ (x,u(x)): x\in\BB^n\big\}
\]
is a minimal hypersurface of $\RR^{n+1}$.

Area bounds for minimal graphs play a key role in the theory of minimal surfaces.
Since $\Si_u$ is a global minimizer of the area
functional among the surfaces with fixed boundary, if $u$ is bounded
on $\BB^n$ it is clear that the $n$-dimensional
Hausdorff measure  $\cH^n(\Si_u)$ is bounded by a non-uniform
constant that depends on the oscillation of~$u$, $\textrm{osc } u:=\max_{\BB^n} u-\min_{\BB^n}
u$.
On the contrary, if we just specify  the value of the oscillation of $u$, we can consider a hyperplane of slope $\frac {\textrm{osc } u}2$, which is a minimal hypersurface whose area $\cH^n(\Si_u)$ is bounded from below by a constant increasing as ${\textrm{osc } u}$. Therefore,  there is not a uniform estimate holding for the area of minimal graphs on $\BB^n$.

Our objective in this note is to explore the non existence of a
higher-codimension analog of the uniform estimate. We will be interested in bounds for
the ${(n-1)}$-dimensional Hausdorff measure of the hypersurface, or
rather of its transverse intersection with a hyperplane. To this end, let us
denote by $\Pi$ the portion of the horizontal hyperplane that is
contained in the unit ball of $\RR^{n+1}$:
\[
\Pi:=\{x\in\RR^{n+1}:x_{n+1}=0\,,\; |x|<1\}\,.
\]

Actually, our goal is to show a more general result claiming that we can prescribe the geometry of the   intersection of a minimal graph on $\BB^n$ with $\Pi$, up to a small deformation:

\begin{theorem}\label{T.level}
Let $S$ be a compact, connected, properly embedded, orientable
hypersurface of $\BB^n$ with non\-empty boundary.
Then, for any integer~$k$ and $\ep>0$, there is a minimal
graph over the unit $n$-ball and an open subset $\Pi'\subset\Pi$ such that the intersection $\Si_u\cap \Pi'$ is given by $\Phi(S)$, where $\Phi:\Pi\to\Pi$ is a
diffeomorphism with $\|\Phi-\id \|_{C^k}<\ep$.
\end{theorem}

If one chooses $S$ to be a compact hypersurface
of $\BB^n$ with area $\cH^{n-1}(S)>c$ and $\ep$ is small enough, the immediate corollary is a codimension 1 analog of the existence of minimal graphs with arbitrarily large area:

\begin{corollary}\label{T.bound}
The $(n-1)$-dimensional measure of the intersection of a minimal graph
over the unit $n$-ball with a hyperplane is not uniformly
bounded. Specifically, given any constant~$c$, there is some~$u$
satisfying Equation~\eqref{eq:MS} for which
$\Si_u$ and~$\Pi$ intersect transversally but
\[
\cH^{n-1}(\Si_u\cap \Pi)>c\,.
\]
\end{corollary}

These minimal graphs with micro-oscillations play the opposite role with respect to area bounds that hyperplanes. Even hyperplanes with arbitrarily large  $n$-measure do not have larger $(n-1)$-measure of its intersection with $\Pi$  than the diameter of the ball. Our construction of minimal graphs leads to arbitrarily large $(n-1)$-measure of the transverse intersection but with $n$-measure less than twice the area of the ball.

The key point of Theorem \ref{T.level} is that it can be analyzed in the linear regime of the minimal surface equation. In fact, the strategy
that we have used to prove it (see Section~\ref{S.proof}) is to construct  harmonic
functions~$v$ on the ball that are small in a $C^k$~norm and whose
zero set $v^{-1}(0)$ contains the hypersurface $S$ up to a small diffeomorphism. The smallness assumption then
permits to promote them to solutions of the minimal surface equation
through an iterative procedure than does not change much the geometry of
the zero set.
Hence, in Corollary \ref{T.bound}, the large $(n-1)$-measure of the
intersection of the minimal graph $\Si_u$ with the hyperplane $\Pi$
comes from micro-oscillations that do not significantly contribute to
the curvature of the minimal hypersurface (which is almost flat).


To conclude, it is worth mentioning that Theorem~\ref{T.level} remains valid when we consider the intersection of a minimal hypersurface with the portion $\Pi'$ of any (non-vertical) hyperplane contained in $\BB^n\times \RR$ (in this case, the minimal hypersurface can be constructed as a graph over $\Pi'$).


\section{Proof of the main theorem}
\label{S.proof}

In this section we will prove Theorem~\ref{T.level}. For this, a well
known result of Whitney ensures that, by perturbing~$S$ a little if
necessary, one can assume that~$S$ is analytic. Now let us consider an
open extension $S'$ of~$S$ (that is, an open, connected, analytic hypersurface $S'$ of $\BB^n$
containing~$S$) and let us denote
by $\Om$ a small neighborhood of the
hypersurface~$S$ whose closure is contained in~$\BB^n$ and such that
$\RR^n\setminus\Om$ is connected. Such a choice of~$S'$ and~$\Om$ is
always possible because $S$ is connected and its boundary is
nonempty.

An important ingredient in the proof of the main theorem is the
construction of a harmonic function on $\RR^n$ for which a small
deformation of~$S'$ is a structurally stable (portion of a) connected
component of its zero set (similar to the construction in \cite{AMPA2013}):

\begin{lemma}\label{L.harmonic}
For any $\ep>0$ there is a harmonic function~$v$  on $\RR^n$ and some
$\de>0$ such that the zero set $u^{-1}(0)$ of any function~$u$ with
$\|u-v\|_{C^k(\Om)}<\de$ satisfies
\[
u^{-1}(0)\cap\Om'=\Psi(S')\,,
\]
where $\Om'$ is an open subset of $\Om$ and $\Psi:\RR^n\to\RR^n$ is a
diffeomorphism with $\|\Psi-\id \|_{C^k(\RR^n)}<\ep$.
\end{lemma}

\begin{proof}
Let us choose an orientation of~$S'$ and denote by $\nu$ the
corresponding unit normal vector. A natural way to define a harmonic function
associated with $S'$ and with some control on its zero set and on its gradient is via the following Cauchy problem:
\begin{equation}\label{eq:cauchyprob}
\De \tv =0, \qquad \tv \big|_{S'}=0, \qquad \frac{\pa\tv}{\pa \nu} \bigg|_{S'}=1.
\end{equation}
The Cauchy--Kowaleskaya theorem ensures the existence of a solution
~$\tv$ of the above problem in a small neighborhood of ~$S'$, which can be taken to
be~$\Om$ without any loss of generality.
Since $\RR^n\setminus \Om$ is connected, the Lax--Malgrange
approximation theorem \cite{Browder} ensures the existence of a harmonic function ~$v:\RR^n\rightarrow \RR$ such that
\begin{equation*}
\| v-\tv \|_{C^k(\Om)}<\de,
\end{equation*}
where $\de$ is a small quantity to be specified later.

Now let $u$ be close to~$v$ in the sense that $\|u-v\|_{C^k(\Om)}<\de$.  Then
\begin{equation*}
\|u-\tv\|_{C^k(\Om)}<2\de\,,
\end{equation*}
so, since $S$ is a component of the nodal set of $\tv$ and the gradient of
$\tv$ does not vanish by \eqref{eq:cauchyprob}, Thom's isotopy theorem
\cite[~Theorem ~20.2]{AR} implies that, for small enough~$\de$, there is an open subset $\Om'$ of
$\BB^n$ and a diffeomorphism $\Psi$
of $\RR^n$  with $\|\Psi-\id\|_{C^k(\RR^n)}<\ep$ such that
$\Psi(S')=u^{-1}(0)\cap\Om'$. The lemma then follows.
\end{proof}

The observation now is that one can construct a solution to the minimal
graph equation on the ball whose zero set is a small perturbation of
that of the harmonic function constructed in the previous lemma. More
precisely, we have the following:

\begin{lemma}\label{L.promote}
Given any  $\de>0$, there is a function~$u$ satisfying the minimal surface
equation~\eqref{eq:MS} in $\BB^n$ and a positive constant~$\la$ such
that  $\|\la u-v\|_{C^k(\BB^n)}<\de$.
\end{lemma}

 \begin{proof}
Assuming that $k\geq2$ without loss of generality and taking any
$\al\in (0,1)$, let us define a function $F:C^{k,\al}(\BB^n)\to C^{k-2,\al}(\BB^n)$ as
\begin{equation*}
F(u):=\frac{1}{2}\nabla u\cdot\nabla \log\big(1+|\nabla u|^2\big)\,.
\end{equation*}
Equation \eqref{eq:MS} is then expressible as
 \begin{equation}\label{eq:MS2}
\De u-F(u)=0\,.
\end{equation}

Let $v$ be the harmonic function on $\RR^n$ that we constructed in
Lemma~\ref{L.harmonic}.
Take a small positive constant $\ep$ that will be fixed later and consider the iterative scheme
\begin{align}
  u_0&:=\ga v    \nonumber  \\
  u_{j+1}&:=\ga v+w_j       \label{defuvw}
\end{align}
where
\begin{equation*}
\ga:=\frac{\ep}{2\|v\|_{C^{k,\al}(\BB^n)}}
\end{equation*}
and the function $w_j$ is the unique solution to the boundary value problem
\begin{align}
\De w_j=F(u_j)\quad  \text{ in } \BB^n   \,, \qquad \label{eq:defwj}
w_j=0 \quad \text{ on } \pa\BB^n.
\end{align}

Our goal is to show that, for small enough $\ep$, $u_j$ converges in $C^{k,\al}(\BB^n)$ to a function $u$ that satisfies the minimal graph equation ~\eqref{eq:MS2}
in $\BB^n$ and is close to $\ga v$ in a suitable sense. To this end,
let us start by noticing that, as an application of the maximum
principle to the boundary problem \eqref{eq:defwj}, the functions $w_j$ must satisfy
\begin{equation*}
\|w_j\|_{C^0 (\BB^n)}\leq C\|F(u_j)\|_{C^0 (\BB^n)}.
\end{equation*}
Standard elliptic estimates then yield
 \begin{align*}
\|w_j\|_{C^{k,\al}(\BB^n)}&\leq C\big(\|w_j\|_{C^0
  (\BB^n)}+\|F(u_j)\|_{C^{k-2,\al}(\BB^n)}\big)\\
& \leq C\big(\|F(u_j)\|_{C^0
  (\BB^n)}+\|F(u_j)\|_{C^{k-2,\al}(\BB^n)}\big)\\
&\leq C \|F(u_j)\|_{C^{k-2,\al} (\BB^n)}\,.
\end{align*}

On the other hand, if we assume that $\|u_j\|_{C^{k,\al}} <\ep$, one
can exploit the above estimate to infer in Equation~\eqref{defuvw} that
\begin{align}
\|u_{j+1}\|_{C^{k,\al}(\BB^n)}&\leq \ga
\|v\|_{C^{k,\al}(\BB^n)}+\|w_j\|_{C^{k,\al}(\BB^n)}\notag\\
&\leq \frac{\ep}{2} + C\|F(u_j)\|_{C^{k-2,\al}(\BB^n)} \nonumber\\
&\leq\frac{\ep}{2}+C\|u_j\|_{C^{k,\al}(\BB^n)}^3 \notag\\
&\leq \frac{\ep}{2}+C\ep^3<\ep\,,\label{eq:induction}
\end{align}
so the norm of $u_{j+1}$ is less than~$\ep$ too. Here we have used that
\[
\|F(w)\| _{C^{k-2,\al}(\BB^n)}\leq C\|w\|_{C^{k,\al}(\BB^n)}^3
\]
and the fact that $\ga \|v\|_{C^{k,\al}(\BB^n)}$ and $C\ep^3$ are
bounded above by
$\ep/2$. Notice, in particular, that since the first function $u_0$ of the iteration satisfies
\begin{equation*}
\|u_0\|_{C^{k,\al}(\BB^n)}\leq \frac{\ep}{2}\,,
\end{equation*}
the induction argument \eqref{eq:induction} then implies that
\begin{equation}\label{eq:estuj}
\|u_j\|_{C^{k,\al}(\BB^n)}<\ep
\end{equation}
for all $j$.

To estimate the difference $u_{j+1}-u_j$, let us use the
bound~\eqref{eq:estuj} to write
\begin{align*}
\|F(u_j)-F(u_{j-1})\|_{C^{k-2,\al}(\BB^n)}&\leq
C\big(\|u_j\|_{C^{k,\al}(\BB^n)}^2+\|u_{j-1}\|_{C^{k,\al}(\BB^n)}^2\big)\|u_j-u_{j-1}\|_{C^{k,\al}(\BB^n)}\\
&\leq C\ep^2\|u_j-u_{j-1}\|_{C^{k,\al}(\BB^n)}\,.
\end{align*}
Since
\[
\De(u_{j+1}-u_u)=F(u_j)-F(u_{j-1})\quad \text{in }\BB^n\,,\qquad
u_{j+1}-u_j=0\quad\text{on }\pd\BB^n\,,
\]
standard elliptic estimates then yield
\begin{align}\label{eq:estdiff} 
\|u_{j+1}-u_j\|_{C^{k,\al}(\BB^n)}&\leq C\|F(u_j)-F(u_{j-1})\|_{C^{k-2,\al}(\BB^n)} \nonumber\\
&<C\ep^2\| u_j-u_{j-1}\|_{C^{k,\al}(\BB^n)}\,.
\end{align}
Taking $\ep$ small enough for $C\ep^2<1$, we infer from
\eqref{eq:estuj} and \eqref{eq:estdiff} that, as $j\to\infty$, $u_j$ converges in $C^{k,\al}(\BB^n)$ to some function $u$ with
\begin{equation}\label{eq:estu}
\|u\|_{C^{k,\al}(\BB^n)}\leq \ep.
\end{equation}
Since the sequence $w_j$ converges to~$w$ in $C^{k,\al}(\BB^n)$, the
function $u$ satisfies the equation
\begin{equation}\label{eq:uinteq}
u=\ga v+w,
\end{equation}
where~$w$ is the unique solution to the problem
\begin{align*}
\De w=F(u) \quad \text{in } \BB^n\,,\qquad w=0 \quad \text{on } \pa\BB^n.
\end{align*}
As $v$ is a harmonic function, one can then take the Laplacian
of~\eqref{eq:uinteq} to show that  $u$ is a solution of the minimal graph equation \eqref{eq:MS} with boundary conditions $u=\ga v$ on $\pa\BB^n$.

Taking $\la:=1/\ga$, one can now use the bound \eqref{eq:estu},
the relation \eqref{eq:uinteq} and the definition of $\ga$ to check that
\begin{align*}
\|\la u-v\|_{C^{k,\al}(\BB^n)}&=\la\|u-\ga
v\|_{C^{k,\al}(\BB^n)}\leq{C\la}\|F(u)\|_{C^{k-2,\al}(\BB^n)}\\
&\leq C\la \|u\|_{C^{k,\al}(\BB^n)}^3\leq C\ep^2<\de
\end{align*}
provided that $\ep$ is sufficiently small.
\end{proof}

 Theorem
\ref{T.level} readily follows from Lemmas \ref{L.harmonic} and \ref{L.promote}.

\section*{Acknowledgments}

The authors are indebted to Joaqu\'\i n P\'erez for his very detailed
explanations about bounds for the area of a minimal graph on the ball. The authors are supported by the ERC Starting Grants~633152 (A.E.\ and
M.A.G.-F.) and~335079
(D.P.-S.). M.A.G.-F. acknowledges the financial support of the Spanish
MINECO through a Severo Ochoa FPI scholarship. This work is supported in part by the
ICMAT--Severo Ochoa grant
SEV-2015-0554.

\bibliographystyle{amsplain}

\end{document}